\documentclass[11pt,draft]{amsart}
\usepackage{amssymb,url}

\theoremstyle{plain}
\newtheorem{theorem}{Theorem}
\newtheorem*{Theorem}{Theorem}
\newtheorem{lemma}{Lemma}
\newtheorem{corollary}{Corollary}
\newtheorem{question}{Question}

\DeclareMathOperator*{\Llim}{L-lim}

\newcommand{\C}{\mathbb{C}}

\newcommand{\N}{\mathbb{N}}
\newcommand{\Z}{\mathbb{Z}}
\newcommand{\R}{\mathbb{R}}

\newcommand{\irB}{\mathcal{B}}

\newcommand{\irH}{\mathcal{H}}
\newcommand{\irK}{\mathcal{K}}
\newcommand{\irL}{\mathcal{L}}
\newcommand{\irM}{\mathcal{M}}
\newcommand{\irN}{\mathcal{N}}

\newcommand{\irX}{\mathcal{X}}
\newcommand{\irY}{\mathcal{Y}}

\begin{document}

\title[]{Asymptotic limits of operators similar to normal operators}
\author{Gy\"orgy P\'al Geh\'er}
\address{Bolyai Institute\\
University of Szeged\\
H-6720 Szeged, Aradi v\'ertan\'uk tere 1, Hungary}
\address{MTA-DE "Lend\"ulet" Functional Analysis Research Group, Institute of Mathematics\\
University of Debrecen\\
H-4010 Debrecen, P.O. Box 12, Hungary}
\email{gehergy@math.u-szeged.hu}
\urladdr{\url{http://www.math.u-szeged.hu/~gehergy/}}

\begin{abstract}
Sz.-Nagy's famous theorem states that a bounded operator $T$ which acts on a complex Hilbert space $\irH$ is similar to a unitary operator if and only if $T$ is invertible and both $T$ and $T^{-1}$ are power bounded. There is an equivalent reformulation of that result which considers the self-adjoint iterates of $T$ and uses a Banach limit $L$. In this paper first we present a generalization of the necessity part in Sz.-Nagy's result concerning operators that are similar to normal operators. In the second part we provide characterization of all possible strong operator topology limits of the self-adjoint iterates of those contractions which are similar to unitary operators and act on a separable infinite-dimensional Hilbert space. This strengthens Sz.-Nagy's theorem for contractions.
\end{abstract}


\maketitle


\section{Introduction}
In this article $\irH$ will denote a complex Hilbert space and the symbol $\irB(\irH)$ will stand for the set of all bounded linear operators acting on $\irH$. We say that $T\in\irB(\irH)$ is a \emph{contraction} if $\|T\| \leq 1$. This article will primarily consider power bounded operators. An operator $T\in\irB(\irH)$ is said to be \emph{power bounded} if $\sup_{n\in\N}\|T^n\| < \infty$ holds.

The study of similarity problems in Hilbert spaces attracted the interest of many mathematicians and it seems to be extremely hard. The first result was given by B. Sz.-Nagy. Namely, in \cite{SzN} he managed to prove a theorem which (even today) belongs to the best known results concerning the study of Hilbert space operators that are similar to normal operators. In the present paper we intend to contribute to this theorem. Regarding this kind of similarity problems, N.-E. Benamara and N. Nikolski provided a resolvent test in \cite{BN} (see also \cite{Kup1, Kup3} for further results on this topic) which became widely known. 

Another type of similarity problems goes back also to Sz.-Nagy and to P. R. Halmos. They raised the following questions: is every power bounded operator similar to a contraction? Is every polynomially bounded operator similar to a contraction? These questions were answered in \cite{Foguel, Lebow, Pisier}. See also \cite{Halmos, Pisier_book}.

Now, we give some auxiliary definitions. The Banach space of bounded complex sequences is denoted by $\ell^\infty(\N)$. We call the linear functional 
\[ L\colon \ell^\infty(\N) \to \C,\quad \underline{x} = \{x_n\}_{n=1}^\infty \mapsto \Llim_{n\to\infty} x_n\] 
a \emph{Banach limit} if the following four conditions are satisfied:
\begin{itemize}
\item $\|L\| = 1$,
\item we have $\Llim_{n\to\infty} x_n = \lim_{n\to\infty} x_n$ for every convergent sequence,
\item $L$ is positive, i.\,e.\,if $x_n\geq 0$ for all $n\in\N$ then $\Llim_{n\to\infty} x_n \geq 0$, and
\item $L$ is shift-invariant, i.\,e.\,$\Llim_{n\to\infty} x_n = \Llim_{n\to\infty} x_{n+1}$.
\end{itemize}
Note that a Banach limit is never multiplicative (see \cite[Section III.7]{Co} for further details).

Let us fix a Banach limit $L$ and consider an arbitrary power bounded operator $T\in\irB(\irH)$. Then the following is a bounded sesqui-linear form
\[ w\colon\irH\times\irH\to\C, \quad w(x,y) := \Llim_{n\to\infty} \langle T^{*n}T^nx,y \rangle. \]
Hence there exists a positive operator $A_{T,L}\in\irB(\irH)$ such that the equation $w(x,y) = \langle A_{T,L}x,y \rangle$ holds for all vectors $x,y\in\irH$. The operator $A_{T,L}$ will be called the \emph{$L$-asymptotic limit} of the power bounded operator $T$, which usually depends on the particular choice of $L$ (see \cite{Ge_matrix}). In the case when $T$ is a contraction, the operator $A_{T,L}$ does not depend on $L$. In fact, in this case the sequence $\{T^{*n}T^n\}_{n=0}^\infty$ is decreasing, therefore it converges to an operator $A_T (= A_{T,L}$ for every Banach limit $L)$ in the strong operator topology (SOT). This positive contraction $A_T$ will be simply called the \emph{asymptotic limit} of $T$. All the possible asymptotic limits of contractions and the $L$-asymptotic limits of power bounded matrices were described by the author in \cite{Ge_contr} and \cite{Ge_matrix}. The present work can be considered a continuation of these two papers.

The $L$-asymptotic limit of a power bounded operator $T$ tells us how the orbits $\big\{\{T^n x\}_{n=1}^\infty \colon x\in\irH\big\}$ behave, since the following holds:
\[ \Llim_{n\to\infty} \|T^n x\|^2 = \|A_{T,L}^{1/2} x\|^2 = \|A_{T,L}^{1/2} T x\|^2 \]
(the equation above is not true in general if we delete the squares). 

The concept of asymptotic limit and their generalizations play important role in the hyperinvariant subspace problem (see e. g. \cite{BK}, \cite{Ke_isom_as}, \cite{Ke_gen_Toep} and \cite{NFBK}). They were also used in many papers concerning other topics. For example E. Durszt proved a generalization of the famous Rota model (see \cite{Ro}) for completely non-unitary contractions in \cite{Du}; a new proof for a Putnam-Fuglede type result was presented in \cite{DuKu} by B. P. Duggal and C. S. Kubrusly; G. Cassier considered similarity problems in \cite{Ca}; and in \cite{Ku_A_proj} it was pointed out how important it is to give several characterizations for the case when the asymptotic limit is idempotent.

Now, we state a well-known reformulation of Sz.-Nagy's theorem (see \cite{SzN} or \cite[Proposition 3.8.]{Ku_Model})

\begin{Theorem}[Reformulation of Sz.-Nagy's theorem]
Consider an arbitrary operator $T\in\irB(\irH)$ and fix a Banach limit $L$. The following three conditions are equivalent:
\begin{itemize}
\item[\textup{(i)}] $T$ is similar to a unitary operator,
\item[\textup{(ii)}] $T$ is power bounded and the $L$-asymptotic limits $A_{T,L}$ and $A_{T^*,L}$ are invertible,
\item[\textup{(iii)}] $T$ has bounded inverse and both $T^{-1}$ and $T$ are power bounded.
\end{itemize} 
Moreover, if we have an arbitrary power bounded operator $T\in\irB(\irH)$, then the next three conditions are also equivalent:
\begin{itemize}
\item[\textup{(i')}] $T$ is similar to an isometry,
\item[\textup{(ii')}] the $L$-asymptotic limit $A_{T,L}$ is invertible,
\item[\textup{(iii')}] there exists a constant $c > 0$ for which the inequality $\|T^n x \| \geq c\|x\|$ holds for every vector $x\in\irH$.
\end{itemize} 
\end{Theorem}

In this paper we investigate whether there is any connection between the asymptotic behaviour of power bounded operators that are similar to each other. In particular we will provide a new property of the $L$-asymptotic limits of operators that are similar to normal ones. This can be considered a generalization of the necessity part in Sz.-Nagy's theorem. After that we will strengthen Sz.-Nagy's theorem, i.\,e.\,we will characterize all the possible asymptotic limits of those contractions that are similar to unitary operators.


\section{Statements of the main theorems}

Before we present the statements of our main results, we need some definitions. The set $\irH_0(T) = \irH_0 := \{x\in\irH\colon \lim_{n\to\infty}\|T^n x\| = 0\}$ will be called the \emph{stable subspace} of $T$. It is well-known that $\ker A_{T,L} = \irH_0$ holds for every Banach limit (see \cite[Theorem 3]{Ke_gen_Toep}) and that $\irH_0$ is a hyperinvariant subspace of $T$. We recall the following classification of power bounded operators. We say that the power bounded operator $T$ is
\begin{itemize}
\item of class $C_{0\cdot}$ if $\irH_0(T) = \irH$, in notation $T\in C_{0\cdot}(\irH)$,
\item of class $C_{1\cdot}$ if $\irH_0(T) = \{0\}$, in notation $T\in C_{1\cdot}(\irH)$,
\item of class $C_{\cdot k}$ ($k\in\{0,1\}$) if $T^*\in C_{k\cdot}(\irH)$, in notation $T\in C_{\cdot k}(\irH)$,
\item of class $C_{jk}$ ($j,k\in\{0,1\}$) if $T\in C_{j\cdot}(\irH) \cap C_{\cdot k}(\irH)$, in notation $T\in C_{jk}(\irH)$.
\end{itemize}

If the operator $A\in\irB(\irH)$ is not zero, then the \emph{reduced minimum modulus} of $A$ is the quantity $\gamma(A) := \inf\{\|Ax\|\colon x\in(\ker A)^\perp, \|x\|=1\}$. If $A$ is a positive operator, then $\gamma(A)>0$ holds exactly when $A$ is the orthogonal sum of a zero and an invertible positive operator. Since the spectral radius of a power bounded operator is at most 1, any normal power bounded operator $N$ is a contraction. Thus $N$ is an orthogonal sum of a unitary operator and a normal contraction which is of class $C_{00}$. This can be easily seen from the functional model of normal operators. Hence the asymptotic limits $A_N$ and $A_{N^*}$ coincide and they are always the orthogonal projections with range $\irH_0^\perp$. It is natural to ask whether the alternative $\gamma(A_{T,L}) > 0$ or $A_{T,L} = 0$ holds for a power bounded operator $T$ which is similar to a normal operator. As we can see from the next theorem, which will be proven in Section \ref{gen_sec}, this is indeed the case.

\begin{theorem}\label{similar_thm}
Let us consider two power bounded operators $T,S\notin C_{0\cdot}(\irH)$ which are similar to each other. Then $\gamma(A_{T,L})>0$ holds for some (and then for all) Banach limits $L$ if and only if $\gamma(A_{S,L})>0$ is valid. 

Moreover, $\gamma(A_{T,L})>0$ holds if and only if the powers of $T$ are bounded from below uniformly on $\irH_0^\perp$, i.e. there exists a constant $c > 0$ such that 
\[ 
c\|x\| \leq \|T^n x\| \quad (x\in\irH_0^\perp, n\in\N).
\]

In particular, if $T$ is similar to a normal operator, then $\gamma(A_{T,L})>0$ and $\gamma(A_{T^*,L})>0$ are satisfied.
\end{theorem}

Theorem \ref{similar_thm} helps us decide whether a given operator is similar to a normal operator. Similarity to other classes can be tested as well. This will be discussed immediately after proving the above theorem. 

In Section \ref{contr_sim_to_U_sec} we will prove the next theorem which provides us further information about the asymptotic limit $A_T$ of a contraction $T\in\irB(\irH)$ that is similar to a unitary operator. We will denote the inner spectral radius of a positive operator $A\in\irB(\irH)$ by $\underline{r}(A) = \min\sigma(A)$. The symbol $\sigma_e(B)$ will denote the \emph{essential spectrum} of an operator $B\in\irB(\irH)$.

\begin{theorem}\label{similar_to_unitary_thm}
Let $\irH$ be an infinite-dimensional Hilbert space and let $T\in\irB(\irH)$ be a contraction which is similar to a unitary operator. Then the relation $\dim\ker(A_T-\underline{r}(A_T) I) \in \{0, \infty\}$ is fulfilled. Consequently, the condition $\underline{r}(A_T)\in \sigma_e(A_T)$ holds.
\end{theorem}

The author showed in \cite[Theorem 6]{Ge_contr} that whenever $T$ is a contraction acting on a separable infinite-dimensional space, then $1\in \sigma_e(A_T)$ or $A_T$ is a finite-rank projection. In the light of Sz.-Nagy's theorem, \cite[Theorem 6]{Ge_contr} and Theorem \ref{similar_to_unitary_thm}, if $\dim\irH = \aleph_0$ and $T\in\irB(\irH)$ is a contraction which is similar to a unitary operator, then necessarily $A_T$ is an invertible, positive contraction and the conditions $1\in\sigma_e(A_T)$ and $\dim\ker(A_T-\underline{r}(A_T) I) \in \{0, \infty\}$ are satisfied. We will prove in Section \ref{contr_sim_to_U_sec} that the converse is also true.

\begin{theorem}\label{similar_to_unitary_rev_thm}
Let $\irH$ be a separable infinite-dimensional Hilbert space, let $A\in\irB(\irH)$ be a positive, invertible contraction and suppose that the conditions $1\in\sigma_e(A)$ and $\dim\ker(A-\underline{r}(A)I) \in \{0,\aleph_0\}$ are fulfilled. Then there exists a contraction $T\in\irB(\irH)$ which is similar to a unitary operator and the asymptotic limit of $T$ is exactly $A$.
\end{theorem}

We note that the $L$-asymptotic limits of power bounded operators that are similar to unitary operators and which act on a finite-dimensional Hilbert space were characterized in \cite{Ge_matrix}. We will close this paper with posing and discussing some questions in Section \ref{fin_rem_ques_sec}.


\section{Generalization of the necessity part in Sz.-Nagy's theorem}\label{gen_sec}

In this section we present the proof of Theorem \ref{similar_thm}. We begin by stating K\'erchy's lemma which will give a significant contribution to this section.

\begin{lemma}[K\'erchy \cite{Ke_isom_as}]\label{Ker_dec_lem}
Consider a power bounded operator $T\in\irB(\irH)$ and the orthogonal decomposition $\irH = \irH_0\oplus\irH_0^\perp$. The block-matrix form of $T$ in this decomposition is
\begin{equation}\label{Ker_dec_eq}
T =
\left(\begin{matrix}
T_0 & R\\
0 & T_1
\end{matrix}\right) \in \irB(\irH_0\oplus\irH_0^\perp)
\end{equation}
where the elements $T_0$ and $T_1$ are of class $C_{0\cdot}$ and $C_{1\cdot}$, respectively.
\end{lemma}

Using this lemma, first we prove the following. We note that the equivalence (i)$\iff$(iii) is a part of Theorem \ref{similar_thm}.

\begin{lemma}\label{Sz.-Nagy_gen_lem}
Consider a power bounded operator $T\notin C_{0\cdot}(\irH)$. Then the following conditions are equivalent:
\begin{itemize}
\item[\textup{(i)}] the inequality $\gamma(A_{T,L}) > 0$ is satisfied for some and then for all Banach limits $L$,
\item[\textup{(ii)}] the compression $T_1 := P_1 T|\irH_0^\perp$ is similar to an isometry, where $P_1$ denotes the orthogonal projection onto the subspace $\irH_0^\perp$,
\item[\textup{(iii)}] the powers of $T$ are bounded from below uniformly on $\irH_0^\perp$, i.e. there exists a constant $c > 0$ such that $c\|x\| \leq \|T^n x\|$ is satisfied on $\irH_0^\perp$ for all $n\in\N$.
\end{itemize}
\end{lemma}

\begin{proof} (i)$\Longrightarrow$(ii). We will use the decomposition $A_{T,L} = 0\oplus A_1 \in\irB(\irH_0\oplus\irH_0^\perp)$, where $A_1$ is obviously invertible. Consider the operator $X_+\in\irB(\irH,\irH_0^\perp)$ which is defined by the equation $X_+ h = A_{T,L}^{1/2} h$ $(h\in\irH)$. The equation $\|X_+h\| = \|X_+Th\|$ implies that
\begin{equation}\label{isom_as_eq}
V X_+ = X_+ T
\end{equation}
holds with a unique isometry $V\in\irB(\irH_0^\perp)$. Now, if we restrict (\ref{isom_as_eq}) to the subspace $\irH_0^\perp$, we get the following
\[ V A_1^{1/2} = VX_+|\irH_0^\perp = X_+ T|\irH_0^\perp = A_1^{1/2}T_1, \]
which verifies that the operator $T_1$ is similar to the isometry $V$.

(ii)$\Longrightarrow$(iii). By Lemma \ref{Ker_dec_lem}, we have
\[ T^n = \left(\begin{matrix}
T_0^n & * \\
0 & T_1^n
\end{matrix}\right). \]
Therefore, by Sz.-Nagy's theorem, there exists a constant $c > 0$ for which
\[ \|T^n x\| \geq \|T_1^n x\| \geq c \|x\| \]
holds for each $n\in\N$ and $x\in\irH_0^\perp$.

(iii)$\Longrightarrow$(i). Let $x\in\irH_0^\perp$ be arbitrary, then we have
\[ \|A_{T,L}^{1/2} x\|^2 = \Llim_{n\to\infty} \|T^n x\|^2 \geq c^2 \|x\|^2, \]
which means exactly that $\gamma(A_{T,L}^{1/2}) \geq c$ and hence $\gamma(A_{T,L}) \geq c^2$ is satisfied.
\end{proof}

Second, we prove the following technical lemma.

\begin{lemma}\label{inv_uper_block-triang_lem}
Consider an orthogonal decomposition $\irH = \irK\oplus\irL$, and an invertible operator $X\in\irB(\irH)$. Suppose that the block-matrix of $X$ is
\[ X = \left(\begin{matrix}
X_{11} & X_{12}\\
0 & X_{22}
\end{matrix}\right) \in \irB(\irK\oplus\irL), \]
and the element $X_{11}\in\irB(\irK)$ is surjective. Then the elements $X_{11}\in\irB(\irK)$ and $X_{22}\in\irB(\irL)$ are invertible and the block-matrix form of $X^{-1}$ is the following:
\[ X^{-1} = \left(\begin{matrix}
X_{11}^{-1} & -X_{11}^{-1}X_{12}X_{22}^{-1}\\
0 & X_{22}^{-1}
\end{matrix}\right) \in \irB(\irK\oplus\irL). \]
\end{lemma}

\begin{proof}
Let $X^{-1} = \left(\begin{matrix}
Y_{11} & Y_{12}\\
Y_{21} & Y_{22}
\end{matrix}\right) \in \irB(\irK\oplus\irL)$. Since $X$ is invertible, $X_{11}$ has to be injective, thus bijective. The (2,1)-element of the block-matrix decomposition of $X^{-1}X = I$ is $Y_{21}X_{11} = 0 \in\irB(\irK,\irL)$ which gives us $Y_{12} = 0$. The (2,2)-elements of $XX^{-1} = I$ and $X^{-1}X = I$ are $X_{22}Y_{22} = I \in\irB(\irL)$ and $Y_{22}X_{22} = I \in\irB(\irL)$, respectively, which imply the invertibility of $X_{22} \in \irB(\irL)$. Finally, an easy calculation verifies the block-matrix form of $X^{-1}$.
\end{proof}

Now we are in a position to present our proof for Theorem \ref{similar_thm}. We note that for any power bounded operator $T\in\irB(\irH)$ and unitary operator $U\in\irB(\irH)$ the equation 
\begin{equation}\label{L-as_lim_uni_eq}
A_{UTU^*,L} = UA_{T,L}U^* 
\end{equation}
holds. In fact, this can be verified directly from the definition of the $L$-asymptotic limit.

\begin{proof}[Proof of Theorem \ref{similar_thm}] We begin with the first part. Let $X\in\irB(\irH)$ be an invertible operator for which $S = XTX^{-1}$ holds. It is easy to see that $\irH_0(S) = X\irH_0(T)$, which gives us $\dim\irH_0(T) = \dim\irH_0(S)$ and $\dim\irH_0(T)^\perp = \dim\irH_0(S)^\perp$. Therefore we can choose a unitary operator $U\in\irB(\irH)$ such that the equation
\begin{equation}\label{sim_pwb_stabel_spc_eq}
\irH_0(T) = U \irH_0(S) = U X \irH_0(T) = \irH_0(USU^*)
\end{equation}
is valid. By (\ref{L-as_lim_uni_eq}), it is enough to prove the inequality 
\[\gamma(A_{USU^*,L}) > 0.\]

Now, consider the block-matrix decompositions (\ref{Ker_dec_eq}) and
\[ UX = \left(\begin{matrix}
Y_{11} & Y_{12}\\
0 & Y_{22}
\end{matrix}\right) \in \irB(\irH_0(T)\oplus(\irH_0(T))^\perp). \]
The latter one is indeed upper block-triangular and moreover, the element $Y_{11}$ is surjective, because of (\ref{sim_pwb_stabel_spc_eq}). Therefore by Lemma \ref{inv_uper_block-triang_lem} we obtain the equation
\[ (UX)^{-1} = \left(\begin{matrix}
Y_{11}^{-1} & -Y_{11}^{-1}Y_{12}Y_{22}^{-1}\\
0 & Y_{22}^{-1}
\end{matrix}\right) \in \irB(\irH_0(T)\oplus(\irH_0(T))^\perp). \]
An easy calculation gives the following:
\[ P_1 USU^*|(\irH_0(T))^\perp = P_1(UX)T(UX)^{-1}|(\irH_0(T))^\perp = Y_{22}T_1Y_{22}^{-1}, \]
where $P_1$ denotes the orthogonal projection onto the subspace $(\irH_0(T))^\perp$. Now, if the inequality $\gamma(A_{T,L}) > 0$ holds, then by Lemma \ref{Sz.-Nagy_gen_lem} the operator $T_1$ is similar to an isometry. But this gives that the compression $P_1 USU^*|(\irH_0(T))^\perp$ is also similar to an isometry, and hence by Lemma \ref{Sz.-Nagy_gen_lem} and (\ref{sim_pwb_stabel_spc_eq}) we get that $\gamma(A_{S,L}) > 0$ holds.

The second part was proven in Lemma \ref{Sz.-Nagy_gen_lem}. 

The third part is an easy consequence of the fact that the asymp\-totic limit of a power bounded normal operator $N$ is always idempotent.
\end{proof}

Next we prove a consequence of Theorem \ref{similar_thm}. We recall definitions of some special classes of operators to which the similarity will be investigated in the forthcoming corollary. The operator $T\in\irB(\irH)$ is said to be 
\begin{itemize}
\item \emph{of class $Q$} if $\|Tx\|^2 \leq \frac{1}{2}(\|T^2x\|^2+\|x\|^2)$ holds for every $x\in\irH$,
\item \emph{log-hyponormal} if $\log(T^*T) \geq \log(TT^*)$ is satisfied.
\end{itemize}
An operator $T$ is called \emph{paranormal} if $\|T x\|^2 \leq \|T^2 x\|\|x\|$ is valid for all $x\in\irH$. It is quite easy to verify from the arithmetic-geometric mean inequality that every paranormal operator is of class $Q$ as well. 

We say that the operator $T$ has the \emph{Putnam--Fuglede property} (or \emph{PF property} for short) if for any operator $X\in\irB(\irH,\irK)$ and isometry $V\in\irB(\irK)$ for which $TX = XV^*$ holds, the equation $T^*X = XV$ is satisfied as well.

\begin{corollary}\label{exotic_cor}
For a power bounded operator $T\in\irB(\irH)$ and a Banach limit $L$ the following implications are valid:
\begin{itemize}
\item[\textup{(i)}] if $T\notin C_{\cdot 0}(\irH)$ is similar to a power bounded operator that has the PF property, then the condition $\gamma(A_{T^*,L}) > 0$ is fulfilled.
\item[\textup{(ii)}] if $T\notin C_{\cdot 0}(\irH)$ is similar to an operator that is either log-hyponormal or of class $Q$ or paranormal, then the inequality $\gamma(A_{T^*,L}) > 0$ is satisfied.
\end{itemize}
\end{corollary}

\begin{proof}
Theorem 3.2 of \cite{Pa_PF} tells us that the PF property for a power bounded operator $T$ is equivalent to the condition that $T$ is the orthogonal sum of a unitary and a power bounded operator of class $C_{\cdot 0}$. Therefore (i) is an easy consequence of Pagacz's result and Theorem \ref{similar_thm}.

If $T$ is log-hyponormal, then Mecheri's result (see \cite{Me}) implies that $T$ has the PF property, and thus $\gamma(A_{T*,L}) > 0$ holds. 

Finally, let us assume that $T$ is a power bounded operator which also belongs to the class $Q$. We prove that then it is a contraction as well. If $\|Tx\|^2-\|x\|^2 > a > 0$ held for a vector $x\in\irH$, then we would obtain $\|T^2x\|^2-\|Tx\|^2 \geq \|Tx\|^2-\|x\|^2 > a$. By induction we could prove that $\|T^{n+1}x\|^2-\|T^nx\|^2 > a$ would hold for every $n\in\N$. Therefore the inequality $\|T^{n+1}x\|^2-\|x\|^2 > n\cdot a$ would be true, which would imply that $T$ could not be power bounded. Consequently, $T$ has to be a contraction. P. Pagacz showed that a contraction which belongs to the class $Q$, shares the PF property (see \cite{Pa_Wold} and \cite{Ok} for the paranormal case). This gives us that $\gamma(A_{T*,L}) > 0$ is valid, which completes our proof.
\end{proof}

Let us consider an arbitrary operator $B\in\irB(\irH)$ with $\|B\|<1$ and the identity operator $I$ on $\irH$. 
Obviously we have $\gamma(A_{(I\oplus B)^*}) = 1 > 0$, but usually $I\oplus B$ does not share the PF property nor it is a log-hyponormal operator or of class $Q$. 
Thus the points of Corollary \ref{exotic_cor} cannot be equivalent conditions. The same is true for the last part of Theorem \ref{similar_thm}.

We close this section with an application of Lemma \ref{Sz.-Nagy_gen_lem}. If we have an orthonormal base $\{e_n\}_{n=-\infty}^\infty$ in $\irH$ and a bounded two-sided sequence $\{w_k\}_{k\in\Z} \subseteq \C$, then the operator $T\in\irB(\irH)$ is called a \emph{weighted bilateral shift operator} if $T e_k = w_k e_{k+1}$ holds for all $k\in\Z$. It is easy to see that if the weighted bilateral shift operator $T$ is power bounded, then the $L$-asymptotic limit satisfies the equation $A_{T,L} e_k = \big(\Llim_{n\to\infty} \prod_{j=0}^{n} |w_{k+j}|^2\big)\cdot e_k$ for every $k\in\Z$. A weighted bilateral shift operator $T$ is power bounded if and only if the inequality 
\[ \sup \Big\{ \prod_{j=0}^{n} |w_{k+j}| \colon k\in\Z, n\in\N\cup\{0\} \Big\} < \infty \]
is fulfilled  (see \cite[Proposition 2]{Sh}). By Sz.-Nagy's theorem $T$ is similar to a unitary operator exactly when it is power bounded and in addition the following holds:
\[ \inf \Big\{ \prod_{j=0}^{n} |w_{k+j}| \colon k\in\Z, n\in\N\cup\{0\} \Big\} > 0. \]

\begin{corollary}\label{shift_cor}
Let $I$ be an arbitrary set of indices. Consider the orthogonal sum $W = \oplus_{i\in I} W_i \in \irB(\oplus_{i\in I}\irH_i)$ which is power bounded, and each summand $W_i$ is a weighted bilateral shift operator that is similar to a unitary operator. If $W$ is similar to a normal operator, then necessarily it is similar to a unitary operator.
\end{corollary}

\begin{proof} Let us denote the $L$-asymptotic limit of $W_i$ by $A_i$. Since the subspaces $\irH_i$ are invariant for the operators $W^{*n}W^n$ ($i\in I, n\in\N$), we obtain the equation $A_{W,L} = \oplus_{i\in I} A_i$. Since each summand $W_i$ is similar to a unitary operator, the operator $A_i$ is invertible and $A_i^{1/2}W_i = S_iA_i^{1/2}$ holds for every $i\in I$ where $S_i$ denotes a simple (i.\,e.\,unweighted) bilateral shift operator. From the power boundedness of $W$, $\sup\{\|A_i\|\colon i\in I\}<\infty$ follows.

On the one hand, if $\sup\{ \|A_i^{-1}\| \colon i\in I\} < \infty$ is satisfied, then the equation 
\[ W = (\oplus_{i\in I} A_i)^{-1/2}(\oplus_{i\in I} S_i) (\oplus_{i\in I} A_i)^{1/2}\] 
gives that $W$ is similar to a unitary operator. On the other hand, if the inequality fails, then $A_{W,L} = \oplus_{i\in I} A_i$ is not invertible, but injective. By point (i) of Theorem of \ref{similar_thm}, we obtain that in this case $W$ cannot be similar to any normal operator.
\end{proof}


\section{Strengthening of Sz.-Nagy's theorem for contractions}\label{contr_sim_to_U_sec}

We begin this section by proving Theorem \ref{similar_to_unitary_thm}. We say that a subspace $\irL\subseteq\irH$ is reducible for an operator $T\in\irB(\irH)$ if $\irL$ is $T$- and $T^*$-invariant.

\begin{proof}[Proof of Theorem \ref{similar_to_unitary_thm}]
Since $A_T$ is invertible, the inequality $\underline{r} := \underline{r}(A_T) > 0$ is satisfied. It is trivial that if $\underline{r} = 1$, then $A_T = I$, and in this case the statement of the theorem is obviously true. Therefore we may suppose that $\underline{r} < 1$.

We will use the notation $\irM = \ker(A_T-\underline{r} I)$. Assume that the condition $0 < \dim\irM < \infty$ holds. If we set an arbitrary vector $h\in\irM$, then we have
\begin{equation}\label{T_inv_inv_eq}
\underline{r}^{1/2} \|h\| = \|A_T^{1/2} h\| = \|A_T^{1/2} T^{-1} h\| \geq \underline{r}^{1/2} \|T^{-1} h\|
\end{equation}
which implies that the inequality $\|T^{-1} h\| \leq \|h\|$ is fulfilled for every $h\in\irM$. But $T$ is a contraction, therefore $\|T^{-1} h\| = \|h\|$ is fulfilled for every $h\in\irM$. Because of the latter equation and (\ref{T_inv_inv_eq}) we deduce $\|A_T^{1/2} T^{-1} h\| = \underline{r}^{1/2} \|T^{-1} h\|$ which implies that the finite-dimensional subspace $\irM$ is invariant for the operator $T^{-1}$. Since $T$ is bijective, we get that $T^{-1}\irM = \irM$ is fulfilled and the restriction $T|\irM$ is unitary. Since $T$ is a contraction, this implies that $\irM$ is reducing for $T$. On the other hand, $\underline{r} I|\irM = A_T|\irM = I|\irM$ follows from this which is a contradiction.
\end{proof}

Before proving Theorem \ref{similar_to_unitary_rev_thm}, we need the following lemma. We note that the method which will be used here is similar to the one which was used in Section 3 of \cite{Ge_contr}. There operator-weighted unilateral shifts were used and here we use operator-weighted bilateral shifts. This will result in some further complications.

\begin{lemma}\label{rev_lem}
Suppose we have a positive, invertible contraction $A\in\irB(\irH)$, an orthogonal decomposition $\irH = \oplus_{k=-\infty}^\infty \irY_k$ where the subspaces $\irY_k$ are reducing for $A$, and a unitary operator $U\in\irB(\irH)$ such that the following conditions hold:
\begin{itemize}
\item[\textup{(i)}] the equation $U\irY_k = \irY_{k+1}$ is satisfied for all $k\in\Z$,
\item[\textup{(ii)}] we have $\lim_{k\to\infty} \underline{r}(A|\irY_k) = 1$, and
\item[\textup{(iii)}] the inequality $\|A^{1/2} y_k\| \leq \|A^{1/2} U y_k\|$ is fulfilled for every $k\in\Z$ and $y_k\in\irY_k$.
\end{itemize}
Then $T := A^{-1/2}UA^{1/2} \in \irB(\irH)$ is a contraction for which $A_T = A$ holds.
\end{lemma}

\begin{proof}
By (iii) we obtain 
\[ \|T^* y_k\| = \|A^{1/2}U^*A^{-1/2} y_k\| \leq \|A^{1/2}UU^*A^{-1/2} y_k\| = \|y_k\|, \]
which gives that $T$ is a contraction.

Consider an arbitrary vector $y_k\in\irY_k$ ($k\in\Z$), the following inequality holds for any $\varepsilon > 0$ choosing $n$ large enough:
\[ \| T^{*n}T^n y_k - A y_k \| = \| A^{1/2}U^{-n}(A^{-1}-I)U^nA^{1/2} y_k \| \]
\[ \leq \| A^{1/2}\| \cdot \|(A^{-1}-I)|\irY_{k+n}\| \cdot \|A^{1/2} y_k\| \leq \varepsilon \cdot \|y_k\|. \]
This shows that $T^{*n}T^n y_k \to A y_k$ holds for every vector $y_k\in\irY_k$ and number $k\in\Z$. But $\irY_k$ is reducing for the operator $T^{*n}T^n$ ($k\in\Z, n\in\N$) which implies that the equation $A_T = A$ is valid.
\end{proof}

Now, we are in a position to present our proof concerning Theorem \ref{similar_to_unitary_rev_thm}. For two real numbers $a < b$ the symbol $]a,b[$ will stand for the open interval with endpoints $a$ and $b$, by $[a,b[$ and $]a,b]$ we will denote the half open-closed intervals, and $[a,b]$ will refer to the closed interval.

\begin{proof}[Proof of Theorem \ref{similar_to_unitary_rev_thm}] Throughout the proof $\irH_A(\omega)$ denotes the spectral subspace of $A$ associated with the Borel subset $\omega\subseteq\R$. Clearly if $T\in\irB(\irH)$ is a contraction, then $A_{T\oplus I} = A_T\oplus A_I = A_T\oplus I \in\irB(\irH\oplus\irH')$. It is also obvious that the conditions in the statement of the theorem concerning $A$ and $A\oplus I$ are simultaneously satisfied whenever the summand $I$ acts on a finite dimensional space. Therefore throughout the proof we shall assume without loss of generality that $\dim\ker(A-I) \in \{0,\aleph_0\}$ and $\underline{r}:=\underline{r}(A) < 1$. We choose an arbitrary two-sided sequence $\{a_k\}_{k=-\infty}^\infty \subseteq ]\underline{r},1[$ such that $a_k < a_{k+1}$ $(k\in\Z)$, $\lim_{k\to -\infty} a_k = \underline{r}$ and $\lim_{k\to \infty} a_k = 1$ are satisfied. We set 
\[
\irX_k := \irH_A([a_k,a_{k+1}[) \quad (k\in\Z),
\]
\[
\irM :=  \ker(A-\underline{r}I) = \sum_{k=-\infty}^{-1} \oplus \irM_k, \; \dim \irM_k = \dim\irM \in\{0,\aleph_0\}\; (k < 0)
\]
and
\[
\irN := \ker(A-I) = \sum_{k=1}^\infty \oplus \irN_k, \; \dim \irN_k = \dim\irN \in\{0,\aleph_0\}\; (k > 0). 
\]
According to the possibilities whether there are infinitely many positive/negative indices such that $\dim\irX_k = \aleph_0$ holds, we may assume, by choosing an appropriate subsequence if necessary, that the following conditions hold:
\begin{itemize}
\item $\dim\irX_k < \aleph_0$ for every $k>0$ or $\dim\irX_k = \aleph_0$ for all $k>0$, and
\item $\dim\irX_k < \aleph_0$ for every $k<0$ or $\dim\irX_k = \aleph_0$ for all $k<0$.
\end{itemize}
Our aim is to apply Lemma \ref{rev_lem}.

First we define the $\irY_k$ subspaces for positive indices and the unitary operator $U$ on these subspaces. There are two different cases.\\
Case 1. If $\dim\irX_k = \aleph_0$ for all $k > 0$, then we set $\irY_k = \irX_k\oplus\irN_k \; (k>0)$ and define $U$ on these subspaces in such a way that $U\irN_k = \irN_{k+1}$ and $U\irX_k = \irX_{k+1}\; (k > 0)$ are valid. \\
Case 2. If $\dim\irX_k < \aleph_0$ for all $k > 0$, then there exists an orthonormal base $\{e_{k,l}\colon k>0,l>0\}$ in $\irH([a_1,1])$ such that $A e_{k,l} = \alpha_{k,l} e_{k,l} \; (k,l>0)$ is fulfilled with some positive numbers $\{\alpha_{k,l}\colon k>0,l>0\}$ where $\alpha_{k,l} \leq \alpha_{k+1,l}$ holds for any $k,l\in\N$ and $\lim_{k\to\infty} \inf\{\alpha_{k,l}\colon l\in\N\} = 1$. 
We define $\irY_k = \vee\{e_{k,l}\colon l\in\N\}$ and $U e_{k,l} = e_{k+1,l} \; (k,l > 0)$.\\
Let us point out that $\dim\irY_k = \aleph_0$ $(k<0)$ holds.

Second, we do the same for non-positive indices. Here we have three different cases. We note that, as we shall see, in every case the condition $\dim\irY_k = \aleph_0$ $(k\geq 0)$ is fulfilled.\\
Case 1. If $\dim\irX_k = \aleph_0$ for all $k < 0$, then let
\[
\irY_k = \left\{
\begin{matrix}
\irX_{k-1}\oplus\irM_{k-1} & \text{if } k < 0\\
\irX_{-1}\oplus\irM_{-1}\oplus\irX_0 & \text{if } k = 0
\end{matrix}
\right.
\]
and define $U$ on these subspaces in a way such that $U\irX_{k-1} = \irX_{k}, U\irM_{k-1} = \irM_{k} \; (k < -1)$, $U\irX_{-2} = \irX_{-1}\oplus\irX_{0}$, $U\irM_{-2} = \irM_{-1}$ and $U\irY_0 = \irY_1$. \\
Case 2. If $\dim\irX_k < \aleph_0$ for all $k \leq 0$, then we can find an orthonormal base $\{e_{k,l}\colon k\leq 0,l>0\}$ in $\irH([\underline{r},a_1[)$ such that $A e_{k,l} = \alpha_{k,l} e_{k,l} \; (k\leq 0,l>0)$ is satisfied with some positive numbers $\{\alpha_{k,l}\colon k\leq 0,l>0\}$ where $\alpha_{k-1,l} \leq \alpha_{k,l}$ holds for any $k\leq 0,l>0$. We set $\irY_k = \vee\{e_{k,l}\colon l\in\N\}$ and $U e_{k-1,l} = e_{k,l} \; (k\leq 0,l > 0)$, $U\irY_0 = \irY_1$. \\
Case 3. In case when $\dim\irX_k < \aleph_0$ for all $k < 0$ and $\dim\irX_0 = \aleph_0$, then we can find an orthonormal base $\{e_{k,l}\colon k< 0,l>0\}$ in $\irH([\underline{r},a_0[)$ such that $A e_{k,l} = \alpha_{k,l} e_{k,l} \; (k<0,l>0)$ is satisfied with some positive numbers $\{\alpha_{k,l}\colon k<0,l>0\}$ where $\alpha_{k-1,l} \leq \alpha_{k,l}$ holds for any $k<0,l>0$. We set $\irY_k = \vee\{e_{k,l}\colon l\in\N\}$ for $k<0$ and $\irY_0 = \irX_0$, moreover, we define $U$ on these subspaces such that $U e_{k-1,l} = e_{k,l} \; (k<0,l > 0)$, $U\irY_{-1} = \irY_0$ and $U\irY_0 = \irY_1$. 

With the above choices points (i)--(iii) of Lemma \ref{rev_lem} are satisfied, therefore our proof is complete.
\end{proof}


\section{Final remarks and open questions}\label{fin_rem_ques_sec}

Throughout this section $\irH$ will be a separable, infinite dimensional Hilbert space. It is known that if $W\in\irB(\irH)$ is a normal bilateral weighted shift operator, then it is a constant multiple of a unitary bilateral shift operator (see \cite{Sh} for further information about shift operators). The following question arises naturally and as far as we know it is open.

\begin{question}
What are those weighted bilateral shift operators which are similar to normal operators?
\end{question}

It is quite easy to see that if the weighted bilateral shift $W$ is similar to a normal operator $N\in\irB(\irH)$, then $N$ is cyclic and the scalar-valued spectral measure is rotation-invariant.

The next questions concern $L$-asymptotic limits of power bounded operators which we left open.

\begin{question}
Which positive, invertible operators $A\in\irB(\irH)$ can be the $L$-asymptotic limits of power bounded operators that are similar to unitary operators?
\end{question}

\begin{question}
Which positive operators $A\in\irB(\irH)$ can be the $L$-asymptotic limits of power bounded operators?
\end{question}

It is sure that not every positive, invertible operator $A$ can be obtained in such a way. The main reason can be found in \cite{Ge_matrix} where it was proved that necessarily $\|A_{T,L}\|\geq 1$ whenever $A_{T,L} \neq 0$. Another reason can be given: equation (\ref{isom_as_eq}) shows that if $A_{T,L} = t I$ holds for some $t > 0$, then $T$ is an isometry, hence $t = 1$. The following question arises naturally:

\begin{question}
Suppose that $A\in\irB(\irH)$ is positive, $\sigma_e(A) = \{t\}$ holds with some $t>0$ and $A$ is the $L$-asymptotic limit of a power bounded operator. Then does necessarily $t = 1$ follow?
\end{question}

As far as we know, no counterexamples can be found in the literature. The $t=0$ case is possible, since a finite-rank operator can easily be the $L$-asymptotic limit of a power bounded operator. The simplest example is if we choose a finite rank projection for $T$ (see \cite{Ge_matrix} concerning further examples). However, the following question is open as well.
 
\begin{question}
If $A\in\irB(\irH)$ is injective, compact, positive and it is the $L$-asymptotic limit of a power bounded operator, then does necessarily $A = 0$ follow?
\end{question}


\section*{Acknowledgement} 

The author expresses his thanks to professor Carlos S. Kubrusly, professor L\'aszl\'o K\'erchy, Patryk Pagacz and the anonymous referee for their useful comments which improved this paper.

The author was supported by the "Lend\"ulet" Program (LP2012-46/2012) of the Hungarian Academy of Sciences. 

This research was also supported by the European Union and the State of Hungary, co-financed by the European Social Fund in the framework of T\'AMOP-4.2.4.A/ 2-11/1-2012-0001 'National Excellence Program'.

\bibliographystyle{ams}

\end{document}